\documentclass[10pt]{article}
\usepackage{amsmath,amscd}
\usepackage{amssymb,latexsym,amsthm}
\usepackage{color}
\usepackage[spanish,english]{babel}
\usepackage[pdftex]{hyperref}
\usepackage{amssymb}
\usepackage{color}
\usepackage{amsmath,amsthm,amscd}
\usepackage[latin1]{inputenc}
\usepackage{lscape}
\usepackage{fancyhdr}
\usepackage{amsfonts}
\usepackage{pb-diagram}

\numberwithin{equation}{section}
\newtheorem{theorem}{Theorem}[section]

\newtheorem{definition}[theorem]{Definition}
\newtheorem{proposition}[theorem]{Proposition}
\newtheorem{lemma}[theorem]{Lemma}

\newtheorem{corollary}[theorem]{Corollary}

\theoremstyle{definition}

\newtheorem{example}[theorem]{Example}

\newtheorem{problem}[theorem]{Problem}
\newtheorem{remark}[theorem]{Remark}

\deactivatequoting

\title{\textbf{Algebraic sets, ideals of points and the Hilbert's Nullstellensatz theorem for skew $PBW$ extensions}}
\author{Oswaldo Lezama\\
\texttt{jolezamas@unal.edu.co}
\\ Seminario de Álgebra Constructiva - SAC$^2$\\ Departamento de Matemáticas\\ Universidad Nacional de
Colombia, Sede Bogot\'a}
\date{}
\begin{document}
\maketitle
\begin{abstract}
	
\noindent In this paper we define the algebraic sets and the ideal of points for bijective skew $PBW$ extensions with coefficients in left Noetherian domains. Some properties of affine algebraic sets of commutative algebraic geometry will be extended, in particular, a Zariski topology will be constructed. Assuming additionally that the extension is quasi-commutative with polynomial center and the ring of coefficients is an algebraically closed field, we will prove an adapted version of the Hilbert's Nullstellensatz theorem that covers the classical one. The Gröbner bases of skew $PBW$ extensions will be used for defining the algebraic sets and for proving the main theorem. Many key algebras and rings coming from mathematical physics and non-commutative algebraic geometry are skew $PBW$ extensions.   
	
\tableofcontents
\noindent 

\bigskip


\bigskip

\end{abstract}
%-------------------------------------------------
%-------------------------------------------------

\section{Introduction}

The ring-theoretic version of the Hilbert's Nullstellensatz theorem for skew $PBW$ extensions has been considered in the beautiful paper \cite{reyes-jason}, following the approach given in \cite{McConnell2}. The algebraic characterization of the theorem presented by the authors of \cite{reyes-jason} does not use the notion of variety (Theorem 3.1, \cite{reyes-jason}). Applying the algebraic sets and the ideal of points introduced in Definition \ref{definition4.14} and Theorem \ref{theorem3.4}, we present in this paper the classical version of this important theorem for quasi-commutative bijective skew $PBW$ extensions of algebraically closed fields. Our version covers the Nullstellensatz theorem of commutative algebraic geometry (see \cite{Fulton}, Chapter 1).

The novelty of the paper is concentrated in the third section, more exactly, the new results are Definition \ref{definition4.14} (inspired in the notion of algebraic set given in Section 5 of \cite{Luerssen}), Theorem \ref{theorem3.4}, Corollary \ref{corollary4.5}, Lemma \ref{lemma4.15} and Theorem \ref{Nullstellensatz} (main theorem). Some examples that illustrate the main theorem are presented in Example \ref{example3.13}.  

For completeness, we conclude this introductory section recalling some notions and results related to prime ideals of an arbitrary ring that we will need for the proof of Lemma \ref{lemma4.15} (see \cite{McConnell} and also \cite{Birkenmeier}, Definition 3). In this paper ring means associative ring with unit.

\begin{definition}
	Let $S$ be a ring and $I,P$ be two-sided ideals of $S$, with $P\neq S$.
	\begin{enumerate}
		\item[\rm (i)]
		$P$ is a \textbf{prime ideal} of $S$ if for any left ideals $L,J$ of $S$ the following condition holds:
		$LJ\subseteq P$ if and only if $L\subseteq P$ or $J\subseteq P$.
		\item[\rm (ii)]The \textbf{radical} of $I$, denoted $\sqrt{I}$, is the intersection of all prime ideals of $S$ containing $I$.
		\item[\rm (iii)]An element $a\in S$ is \textbf{$I$-strongly nilpotent} if for any given sequence $\mathcal{S}:=\{a_i\}_{i\geq 1}$ of elements of $S$, with $a_1:=a$ and $a_{i+1}\in a_iSa_i$, there exists $m(\mathcal{S})\geq 1$ such that $a_{m(\mathcal{S})}\in I$. We say that $a$ is \textbf{$I$-nilpotent} if there exists $m\geq 1$ such that $a^m\in I$.
		\item[\rm (iv)]$P$ is \textbf{completely prime} if the following condition holds for any $a,b\in S$: $ab\in P$ if and only if $a\in P$ or $b\in P$.  
		\item[\rm (v)]$P$ is \textbf{completely semiprime} if the following condition holds for any $a\in S$: $a^2\in P$ if and only if $a\in P$. 
	\end{enumerate}
\end{definition}
It is clear that if $a\in S$ is $I$-strongly nilpotent, then $a$ is $I$-nilpotent. If $a\in Z(S)$, then the converse is true. Observe that any element $a\in S$ is $S$-strongly nilpotent and $\sqrt{S}:=S$. If $P$ is completely prime, then $P$ is completely semiprime. By induction on $m$ it is easy to show that $P$ is completely semiprime if and only if the following condition holds: For any $a\in S$ and any $m\geq 1$,  $a^m\in P$ if and only if $a\in P$. 

\begin{proposition}\label{proposition4.14}
	Let $S$ be a ring and $P$ be a proper two-sided ideal of $S$. $P$ is a prime ideal if and only if the following condition holds for any elements $a,b\in S$: $aSb\subseteq P$ if and only if $a\in P$ or $b\in P$.
\end{proposition}

\begin{proposition}\label{proposition4.15}
	Let $S$ be a ring and $P$ be a proper two-sided ideal of $S$. $P$ is a prime ideal if and only if the following condition holds for any left ideals $L,J$ of $S$: If $P\subsetneq L$ and  $P\subsetneq J$, then $LJ\nsubseteq P$.
\end{proposition}

\begin{proposition}
	Let $S$ be a ring and $I$ be a two-sided ideal of $S$. Then,
	\begin{center}
		$\sqrt{I}=\{a\in S\mid a\ \text{is $I$-strongly nilpotent}\}$.
	\end{center}
\end{proposition}
\begin{proof}
	Let $a\in S$ such that $a\notin \sqrt{I}$, then there exists a prime ideal $P$ of $S$, containing $I$, such that $a\notin P$, hence, by Proposition \ref{proposition4.14},  $aSa\nsubseteq P$. This says that there exists $b\in S$ such that $aba\notin P$. Let $a_1:=a$ and $a_2:=aba$. Thus, $a_2Sa_2\nsubseteq P$ and hence there exists $c\in S$ such that $a_2ca_2\notin P$. Let $a_3:=a_2ca_2$. Continuing this way we get a sequence $\{a_i\}_{i\geq 1}$ of elements of $S$ such that $a_i\notin P$ for every $i\geq 1$, hence, $a_i\notin I$ for every $i\geq 1$. This means that $a$ is not $I$-strongly nilpotent.
	
	Conversely, assume that $a\in S$ is not $I$-strongly nilpotent, then there exists a sequence $\mathcal{S}:=\{a_i\}_{i\geq 1}$ of elements of $S$, with $a_1:=a$ and $a_{i+1}\in a_iSa_i$, such that for every $i\geq 1$, $a_i\notin I$. By Zorn's lemma, there exists a two-sided ideal $P$ of $S$, containing $I$, maximal with respect to the condition $\mathcal{S}\cap P=\emptyset$ (observe that $I\supseteq I$ and $\mathcal{S}\cap I=\emptyset$). We will show that $P$ is a prime ideal of $S$. It is clear that $P\neq S$. Let $L,J$ be left ideals of $S$ such that $P\subsetneq L$ and $P\subsetneq J$. Then, $L\cap \mathcal{S}\neq \emptyset$ and $J\cap \mathcal{S}\neq \emptyset$, so there exists $a_i\in L$ and $a_j\in J$. Let $k:=\max\{i,j\}$, then $a_{k+1}\in LJ$, but $a_{k+1}\notin P$, i.e., $LJ\nsubseteq P$. Thus, from Proposition \ref{proposition4.15}, $P\supseteq I$ is a prime ideal such that $a\notin P$, so $a\notin \sqrt{I}$. 
\end{proof}

\section{Skew $PBW$ extensions}

In this section we recall some basic facts about the class of noncommutative rings of polynomial type known as skew $PBW$ extensions. In particular, we include the ingredients of the Gröbner theory of skew $PBW$ extensions needed in the next section. For more details see \cite{Lezama-sigmaPBW}, Chapters 1, 2, 3, 13.

\begin{definition}[\cite{LezamaGallego},\cite{Lezama-sigmaPBW}]\label{gpbwextension}
	Let $R$ and $A$ be rings. We say that $A$ is a \textit{\textbf{skew $PBW$
			extension of $R$}} $($also called a $\sigma-PBW$ extension of
	$R$$)$ if the following conditions hold:
	\begin{enumerate}
		\item[\rm (i)]$R\subseteq A$.
		\item[\rm (ii)]There exist finitely many elements $x_1,\dots ,x_n\in A$ such $A$ is an $R$-free left module with basis
		\begin{center}
			${\rm Mon}(A):= \{x^{\alpha}=x_1^{\alpha_1}\cdots
			x_n^{\alpha_n}\mid \alpha=(\alpha_1,\dots ,\alpha_n)\in
			\mathbb{N}^n\}$, with $\mathbb{N}:=\{0,1,2,\dots\}$.
		\end{center}
		In this case we say that $A$ is a \textbf{ring of left polynomial type} over $R$ with respect to
		$\{x_1,\dots,x_n\}$. The set ${\rm Mon}(A)$ is called the set of \textbf{standard monomials} of
		$A$.
		\item[\rm (iii)]For every $1\leq i\leq n$ and $r\in R-\{0\}$ there exists $c_{i,r}\in R-\{0\}$ such that
		\begin{equation}\label{sigmadefinicion1}
		x_ir-c_{i,r}x_i\in R.
		\end{equation}
		\item[\rm (iv)]For every $1\leq i,j\leq n$ there exists $c_{i,j}\in R-\{0\}$ such that
		\begin{equation}\label{sigmadefinicion2}
		x_jx_i-c_{i,j}x_ix_j\in R+Rx_1+\cdots +Rx_n.
		\end{equation}
		Under these conditions we will write $A:=\sigma(R)\langle
		x_1,\dots ,x_n\rangle$.
	\end{enumerate}
\end{definition}
Associated to a skew $PBW$ extension $A=\sigma(R)\langle x_1,\dots
,x_n\rangle$ there are $n$ injective endomorphisms
$\sigma_1,\dots,\sigma_n$ of $R$ and $\sigma_i$-derivations, as
the following proposition shows.

\begin{proposition}[\cite{LezamaGallego}, Proposition 3]\label{sigmadefinition}
	Let $A$ be a skew $PBW$ extension of $R$. Then, for every $1\leq
	i\leq n$, there exist an injective ring endomorphism
	$\sigma_i:R\rightarrow R$ and a $\sigma_i$-derivation
	$\delta_i:R\rightarrow R$ such that
	\begin{center}
		$x_ir=\sigma_i(r)x_i+\delta_i(r)$,
	\end{center}
	for each $r\in R$.
\end{proposition}

Two remarkable particular cases of skew $PBW$ extensions are recalled next. 

\begin{definition}[\cite{Lezama-sigmaPBW}, Chapter 1]\label{sigmapbwderivationtype}
	Let $A$ be a skew $PBW$ extension.
	\begin{enumerate}
		\item[\rm (a)]
		$A$ is \textbf{quasi-commutative}\index{quasi-commutative} if conditions {\rm(}iii{\rm)} and {\rm(}iv{\rm)} in Definition
		\ref{gpbwextension} are replaced by
		\begin{enumerate}
			\item[\rm ($iii'$)]For every $1\leq i\leq n$ and $r\in R-\{0\}$ there exists a $c_{i,r}\in R-\{0\}$ such that
			\begin{equation}
			x_ir=c_{i,r}x_i.
			\end{equation}
			\item[\rm ($iv'$)]For every $1\leq i,j\leq n$ there exists $c_{i,j}\in R-\{0\}$ such that
			\begin{equation}
			x_jx_i=c_{i,j}x_ix_j.
			\end{equation}
		\end{enumerate}
		\item[\rm (b)]$A$ is \textbf{bijective}\index{bijective} if $\sigma_i$ is bijective for
		every $1\leq i\leq n$ and $c_{i,j}$ is invertible for any $1\leq i,j\leq n$.
	\end{enumerate}
\end{definition}

If $A=\sigma(R)\langle x_1,\dots,x_n\rangle$ is a skew $PBW$
extension of the ring $R$, then, as was observed in Proposition
\ref{sigmadefinition}, $A$ induces injective endomorphisms
$\sigma_k:R\to R$ and $\sigma_k$-derivations $\delta_k:R\to R$,
$1\leq k\leq n$. Moreover, from the Definition
\ref{gpbwextension}, there exists a unique finite set of constants
$c_{ij}, d_{ij}, a_{ij}^{(k)}\in R$, $c_{ij}\neq 0$, such that
\begin{equation}\label{equation1.2.1}
x_jx_i=c_{ij}x_ix_j+a_{ij}^{(1)}x_1+\cdots+a_{ij}^{(n)}x_n+d_{ij},
\ \text{for every}\  1\leq i<j\leq n.
\end{equation}
If $A$ is quasi-commutative, then $\delta_k=0$ for every $1\leq k\leq n$ and $p_{\alpha,r},p_{\alpha, \beta}=0$ in Proposition \ref{coefficientes}.

Many important algebras and rings coming from mathematical physics and non-commutative algebraic geometry
are particular examples of skew $PBW$ extensions: \textbf{Habitual ring of
	polynomials in several variables}, Weyl algebras, enveloping
algebras of finite dimensional Lie algebras, algebra of
$q$-differential operators, many important types of Ore algebras, in particular, single Ore extensions, 
algebras of diffusion type, additive and multiplicative analogues
of the Weyl algebra, dispin algebra $\mathcal{U}(osp(1,2))$,
quantum algebra $\mathcal{U}'(so(3,K))$, Woronowicz algebra
$\mathcal{W}_{\nu}(\mathfrak{sl}(2,K))$, Manin algebra
$\mathcal{O}_q(M_2(K))$, coordinate algebra of the quantum group
$SL_q(2)$, $q$-Heisenberg algebra \textbf{H}$_n(q)$, Hayashi
algebra $W_q(J)$, differential operators on a quantum space
$D_{\textbf{q}}(S_{\textbf{q}})$, Witten's deformation of
$\mathcal{U}(\mathfrak{sl}(2,K))$, multiparameter Weyl algebra
$A_n^{Q,\Gamma}(K)$, quantum symplectic space
$\mathcal{O}_q(\mathfrak{sp}(K^{2n}))$, some quadratic algebras in
3 variables, some 3-dimensional skew polynomial algebras,
particular types of Sklyanin algebras, homogenized enveloping
algebra $\mathcal{A}(\mathcal{G})$, Sridharan enveloping algebra
of 3-dimensional Lie algebra $\mathcal{G}$, among many others. For
a precise definition of any of these rings and algebras see 
\cite{lezamareyes1} and \cite{Lezama-sigmaPBW}. The skew $PBW$ has been intensively studied in the last years (see \cite{Lezama-sigmaPBW}).

Next we will fix some notation and a monomial order in $A$ (see \cite{Lezama-sigmaPBW}, Chapter 1). 

\begin{definition}\label{1.1.6}
	Let $A$ be a skew $PBW$ extension of $R$ with endomorphisms $\sigma_i$ as in Proposition
	\ref{sigmadefinition}, $1\leq i\leq n$.
	\begin{enumerate}
		\item[\rm (i)]For $\alpha=(\alpha_1,\dots,\alpha_n)\in \mathbb{N}^n$,
		$\boldsymbol{\sigma^{\alpha}}:=\sigma_1^{\alpha_1}\cdots \sigma_n^{\alpha_n}$,
		$\boldsymbol{|\alpha|}:=\alpha_1+\cdots+\alpha_n$. If $\beta=(\beta_1,\dots,\beta_n)\in \mathbb{N}^n$, then
		$\boldsymbol{\alpha+\beta}:=(\alpha_1+\beta_1,\dots,\alpha_n+\beta_n)$.
		\item[\rm (ii)]For $X=x^{\alpha}\in \mathrm{Mon}(A)$,
		$\boldsymbol{\exp(X)}:=\alpha$ and $\boldsymbol{\deg(X)}:=|\alpha|$.
		\item[\rm (iii)]Let $0\neq f\in A$. If $t(f)$ is the finite
		set of terms that conform $f$, i.e., if $f=c_1X_1+\cdots +c_tX_t$, with $X_i\in \mathrm{Mon}(A)$ and $c_i\in
		R-\{0\}$, then $\boldsymbol{t(f)}:=\{c_1X_1,\dots,c_tX_t\}$.
		\item[\rm (iv)]Let $f$ be as in {\rm(iii)}, then $\boldsymbol{\deg(f)}:=\max\{\deg(X_i)\}_{i=1}^t.$
	\end{enumerate}
\end{definition}

In $\mathrm{Mon}(A)$ we define
\begin{center}
	$x^{\alpha}\succeq x^{\beta}\Longleftrightarrow
	\begin{cases}
	x^{\alpha}=x^{\beta}\\
	\text{or} & \\
	x^{\alpha}\neq x^{\beta}\, \text{but} \, |\alpha|> |\beta| & \\
	\text{or} & \\
	x^{\alpha}\neq x^{\beta},|\alpha|=|\beta|\, \text{but $\exists$ $i$ with} &
	\alpha_1=\beta_1,\dots,\alpha_{i-1}=\beta_{i-1},\alpha_i>\beta_i.
	\end{cases}$
\end{center}
It is clear that this is a total order on $\mathrm{Mon}(A)$, called \textit{\textbf{deglex}} order. If
$x^{\alpha}\succeq x^{\beta}$ but $x^{\alpha}\neq x^{\beta}$, we write $x^{\alpha}\succ x^{\beta}$.
Each element $f\in A-\{0\}$ can be represented in a unique way as $f=c_1x^{\alpha_1}+\cdots
+c_tx^{\alpha_t}$, with $c_i\in R-\{0\}$, $1\leq i\leq t$, and $x^{\alpha_1}\succ \cdots \succ
x^{\alpha_t}$. We say that $x^{\alpha_1}$ is the \textit{\textbf{leading monomial}} of $f$ and we write
$lm(f):=x^{\alpha_1}$; $c_1$ is the \textit{\textbf{leading coefficient}} of $f$, $lc(f):=c_1$, and
$c_1x^{\alpha_1}$ is the \textit{\textbf{leading term}} of $f$ denoted by $lt(f):=c_1x^{\alpha_1}$. We say that $f$ is \textit{\textbf{monic}} if $lc(f):=1$. If $f=0$,
we define $lm(0):=0$, $lc(0):=0$, $lt(0):=0$, and we set $X\succ 0$ for any $X\in \mathrm{Mon}(A)$. We observe that
\begin{center}
	$x^{\alpha}\succ x^{\beta}\Rightarrow lm(x^{\gamma}x^{\alpha}x^{\lambda})\succ
	lm(x^{\gamma}x^{\beta}x^{\lambda})$, for every $x^{\gamma},x^{\lambda}\in \mathrm{Mon}(A)$.
\end{center}

The next proposition complements Definition \ref{gpbwextension}.

\begin{proposition}[\cite{LezamaGallego},\cite{Lezama-sigmaPBW}]\label{coefficientes}
	Let $A$ be a ring of a left polynomial type over $R$ w.r.t.\ $\{x_1,\dots,x_n\}$. $A$ is a skew
	$PBW$ extension of $R$ if and only if the following conditions hold:
	\begin{enumerate}
		\item[\rm (a)]For every $x^{\alpha}\in \mathrm{Mon}(A)$ and every $0\neq
		r\in R$ there exist unique elements $r_{\alpha}:=\sigma^{\alpha}(r)\in R-\{0\}$ and $p_{\alpha
			,r}\in A$ such that
		\begin{equation}\label{611}
		x^{\alpha}r=r_{\alpha}x^{\alpha}+p_{\alpha , r},
		\end{equation}
		where $p_{\alpha ,r}=0$ or $\deg(p_{\alpha ,r})<|\alpha|$ if $p_{\alpha , r}\neq 0$. Moreover, if
		$r$ is left invertible, then $r_\alpha$ is left invertible.
		
		\item[\rm (b)]For every $x^{\alpha},x^{\beta}\in \mathrm{Mon}(A)$ there
		exist unique elements $c_{\alpha,\beta}\in R$ and $p_{\alpha,\beta}\in A$ such that
		\begin{equation}\label{612}
		x^{\alpha}x^{\beta}=c_{\alpha,\beta}x^{\alpha+\beta}+p_{\alpha,\beta},
		\end{equation}
		where $c_{\alpha,\beta}$ is left invertible, $p_{\alpha,\beta}=0$ or
		$\deg(p_{\alpha,\beta})<|\alpha+\beta|$ if $p_{\alpha,\beta}\neq 0$.
	\end{enumerate}
\end{proposition}

We conclude this subsection recalling some of the main ingredients of the Gröbner theory of skew $PBW$ extensions, namely, the Division Algorithm and the notion of Gröbner basis of a left ideal of $A$. For all details see \cite{Lezama-sigmaPBW}, Chapter 13. For the condition (ii) in Definition \ref{reductionsigmapbw} below, some natural computational conditions on $R$ will be assumed.
\begin{definition}\label{LGSring}
	A ring $R$ is \textbf{left Gr\"obner soluble}\index{left Grobner soluble@left Gr\"obner soluble} {\rm(}$LGS${\rm)}\index{LGS ring@$LGS$ ring} if the
	following conditions hold:
	\begin{enumerate}
		\item[\rm (i)]$R$ is left noetherian.
		\item[\rm (ii)]Given $a,r_1,\dots,r_m\in R$ there exists an
		algorithm which decides whether $a$ is in the left ideal
		$Rr_1+\cdots+Rr_m$, and if so, finds $b_1,\dots,b_m\in R$ such that
		$a=b_1r_1+\cdots+b_mr_m$.
		\item[\rm (iii)]Given $r_1,\dots,r_m\in R$ there exists an
		algorithm which finds a finite set of generators of the left
		$R$-module
		\begin{center}
			$\mathrm{Syz}_R[r_1\ \cdots \ r_m]:=\{(b_1,\dots,b_m)\in
			R^m\mid b_1r_1+\cdots+b_mr_m=0\}$.
		\end{center}
	\end{enumerate}
\end{definition}

\begin{definition}
	Let $x^{\alpha},x^{\beta}\in \mathrm{Mon}(A)$. We say that $x^{\alpha}$
	\textbf{divides}\index{divides} $x^{\beta}$, denoted by $x^{\alpha}\mid x^{\beta}$, if there exists a unique $x^{\theta}\in \mathrm{Mon}(A)$ such that
	$x^{\beta}=lm(x^{\theta}x^{\alpha})=x^{\theta+\alpha}$ and hence
	$\beta=\theta+\alpha$.
\end{definition}

\begin{definition}\label{reductionsigmapbw}
	Let $F$ be a finite set of nonzero elements of $A$, and let
	$f,h\in A$. We say that $f$ \textbf{reduces to $h$ by $F$ in one step},\index{reduces in one step}
	denoted $f\xrightarrow{\,\, F\,\, } h$, if there exist elements
	$f_1,\dots,f_t\in F$ and $r_1,\dots,r_t\in R$ such that
	\begin{enumerate}
		\item[\rm (i)]$lm(f_i)\mid lm(f)$, $1\leq i\leq t$, i.e., there exists an
		$x^{\alpha_i}\in \mathrm{Mon}(A)$ such that
		$lm(f)=lm(x^{\alpha_i}lm(f_i))$, i.e.,
		$\alpha_i+\exp(lm(f_i))=\exp(lm(f))$.
		\item[\rm
		(ii)]$lc(f)=r_1\sigma^{\alpha_1}(lc(f_1))c_{\alpha_1,f_1}+\cdots+r_t\sigma^{\alpha_t}(lc(f_t))c_{\alpha_t,f_t}$,
		where $c_{\alpha_i,f_i}$ are defined as in Proposition
		\ref{coefficientes}, i.e.,
		$c_{\alpha_i,f_i}:=c_{\alpha_i,\exp(lm(f_i))}$.
		\item[\rm (iii)]$h=f-\sum_{i=1}^tr_ix^{\alpha_i}f_i$.
	\end{enumerate}
	We say that $f$ \textbf{reduces}\index{reduces} to $h$ by $F$, denoted $f\xrightarrow{\,\,
		F\,\, }_{+}h$, if there exist $h_1,\dots ,h_{t-1}\in A$ such that
	\begin{center}
		$\begin{CD} f @>{F}>> h_1 @>{F}>> h_2 @>{F}>>\cdots @>{F}>>h_{t-1}
		@>{F}>>h.
		\end{CD}$
	\end{center}
			$f$ is \textbf{reduced}\index{reduced} {\rm(}also called \textbf{minimal}{\rm)}\index{minimal} w.r.t.\ $F$ if $f =
	0$ or there is no one step reduction  of $f$ by $F$, i.e., one of
	the conditions $(i)$ or $(ii)$
	fails. Otherwise, we will say that $f$ is \textbf{reducible}\index{reducible} w.r.t.\ $F$. If
	$f\xrightarrow{\,\, F\,\, }_{+} h$ and $h$ is reduced w.r.t.\ $F$,
	then we say that $h$ is a \textbf{remainder}\index{remainder} for $f$ w.r.t.\ $F$.
\end{definition}
By definition we will assume that $0\xrightarrow {F} 0$.

\begin{proposition}[Division algorithm]\label{algdivforPBW}
	Let $F=\{f_1,\dots ,f_t\}$ be a finite set of nonzero polynomials
	of $A$ and $f\in A$, then there exist
	polynomials $q_1,\dots ,q_t,h\in A$, with $h$ reduced w.r.t. $F$,
	such that $f\xrightarrow{\,\, F\,\, }_{+} h$ and
	\[
	f=q_1f_1+\cdots +q_tf_t+h,
	\]
	with
	\[
	lm(f)=\max\{lm(lm(q_1)lm(f_1)),\dots
	,lm(lm(q_t)lm(f_t)),lm(h)\}.
	\]
\end{proposition}
\begin{definition}\label{definition3.10}
	Let $I\neq 0$ be a left ideal of $A$ and
	let $G$ be a nonempty finite subset of nonzero polynomials of
	$I$. $G$ is a \textbf{Gröbner basis} \index{Grobner basis@Gr\"obner basis} for $I$ if each element
	$0\neq f\in I$ is reducible w.r.t.\ $G$.
\end{definition}
\begin{proposition}\label{153}
	Let $I\neq 0$ be a left ideal of $A$. Then,
	\begin{enumerate}
		\item[\rm(i)]If $G$ is a Gröbner basis for $I$, then $I=\langle G\}$ $($the left ideal of $A$ generated by $G$$)$.
		\item[\rm(ii)]Let $G$ be a Gröbner basis for $I$. If $f\in I$ and
		$f\xrightarrow{\,\, G\,\, }_{+} h$, with $h$ reduced, then $h=0$.
		\item[\rm(iii)]Let $G=\{g_1,\dots,g_t\}$ be a set of nonzero polynomials of $I$ with $lc(g_i)\in R^{*}$ for each $1\leq i\leq t$. Then,
		$G$ is a Gröbner basis of $I$ if and only if given $0\neq r\in I$
		there exists an $i$ such that $lm(g_i)$ divides $lm(r)$.
	\end{enumerate}
\end{proposition}

\begin{remark}\label{reamrk3.12}
	(i) We remark that the Gröbner theory of skew $PBW$ extensions and some of its important applications in homological algebra have been implemented in Maple in
	\cite{Fajardo2} and \cite{Fajardo3} (see also \cite{Lezama-sigmaPBW}). This implementation is based
	on the library \textbf{\texttt{SPBWE.lib}} specialized for working with bijective skew $PBW$
	extensions. The library has utilities to calculate Gröbner bases, and
	it includes some functions that compute the module of syzygies,
	free resolutions and left inverses of matrices, among other things. For the implementation was assumed that
	$A=\sigma(R)\langle x_1,\dots,x_n\rangle$ is a bijective skew $PBW$ extension of an $LGS$
	ring $R$ and $\mathrm{Mon}(A)$ is endowed with some monomial order $\succeq$.	
	
	(ii) \textbf{From now on in this paper we will assume that $\boldsymbol{A:=\sigma(R)\langle
			x_1,\dots ,x_n\rangle}$ is a bijective skew $\textbf{\textit{PBW}}$ extension of $\boldsymbol{R}$, where $\boldsymbol{R}$ is a left noetherian domain}. This implies that $A$ is a left noetherian domain (see \cite{Lezama-sigmaPBW}, Proposition 3.2.1 and Theorem 3.1.5: \textit{Hilbert's basis theorem for skew $PBW$ extensions}). In the examples where we use the library \texttt{SPBWE.lib} we have assumed additionally that $R$ is $LGS$. This implies that $A$ is $LGS$ (see \cite{Lezama-sigmaPBW}, Chapter 15).   
\end{remark}

\section{Algebraic sets and ideals of points for skew $PBW$ extensions}

This last section represents the novelty of the present work. We will study for the skew $PBW$ extensions the algebraic sets, the ideal of points and the relationship between them. Some properties of the affine algebraic sets of commutative algebraic geometry (see \cite{Fulton}, Chapter 1) will be extended in this section. In particular, we will prove a result (Theorem \ref{Nullstellensatz}) about an adapted version of the classical Hilbert's Nullstellensatz theorem of the commutative algebraic geometry for quasi-commutative skew $PBW$ extensions over algebraically closed fields and with polynomial center. The result covers the classical one.

We will assume on $A$ the conditions in (ii) of Remark \ref{reamrk3.12}. 

\subsection{Roots of polynomials}

For $n\geq 1$, let $R^n$ be the left $R$-module of vectors over $R$ of $n$ components. Let $f\in A$ and $Z:=(z_1,\dots,z_n)\in R^n$. By Proposition \ref{algdivforPBW},there exist
polynomials $q_1,\dots ,q_t,h\in A$, with remainder $h$ reduced w.r.t.\ $F:=\{x_1-z_1,\dots,x_n-z_n\}$,
such that $f\xrightarrow{\,\, F\,\, }_{+} h$ and
\[
f=q_1(x_1-z_1)+\cdots +q_n(x_n-z_n)+h.
\]
In general, $h$ is not unique, and even worse, it could not belong to $R$, as the next example shows.

\begin{example}
	Consider the Witten algebra (see \cite{Lezama-sigmaPBW}, Chapter 2) $A:=\sigma(\mathbb{Q})\langle x,y,z\rangle$ defined by
	\begin{center}
		$zx = xz-x$, $zy = yz+2y$, $yx = 2xy$.
	\end{center}
	For $Z:=(1,-2-3)\in \mathbb{Q}^3$ and $f:=x^2y+xz+yz\in A$, with \textbf{\texttt{SPBWE.lib}} the Algorithm Division produces
	\begin{center}
		$f=(\frac{1}{2}xy+\frac{1}{4}y)(x-1)+\frac{1}{4}(y+2)+0(z+3)+xz+yz-\frac{1}{2}$, 
	\end{center}
	i.e., $q_1=\frac{1}{2}xy+\frac{1}{4}y$, $q_2:=\frac{1}{4}$, $q_3=0$ and $h=xz+yz-\frac{1}{2}$. 
	
	Even for quasi-commutative skew $PBW$ extensions the situation is similar. In fact, consider a $3$-multiparametric quantum space (see \cite{Lezama-sigmaPBW}, Chapter 4) $A:=\sigma(\mathbb{C})\langle x,y,z\rangle$ defined by
	\begin{center}
		$yx = 2ixy$, $zx=3ixz$, $zy=-iyz$.
	\end{center}
	For $Z:=(i,2i,3i)\in \mathbb{C}^3$ and $f:=x^2y+yz^2+xz\in A$, with \textbf{\texttt{SPBWE.lib}} we found that
	\begin{center}
		$f=(\frac{1}{2}ixy-\frac{1}{4}iy)(x-i)+\frac{1}{4}(y-2i)+0(z-3i)+yz^2+xz+\frac{1}{2}i$, 
	\end{center}
	i.e., $q_1=\frac{1}{2}ixy-\frac{1}{4}iy$, $q_2:=\frac{1}{4}$, $q_3=0$ and $h=yz^2+xz+\frac{1}{2}i$.  
\end{example}  

Thus, the evaluation of a polynomial $f\in A$ in a given $Z\in R^n$ as the remainder in the Division Algorithm is not a good idea. However, the following notion does not depend on the Division Algorithm. 
\begin{definition}
	Let $n\geq 1$, $f\in A$ and $Z:=(z_1,\dots,z_n)\in R^n$. $Z$ is a \textbf{root} of $f$ if and only if $f$ is in the two-sided ideal generated by $x_1-z_1,\dots, x_n-z_n$. This condition is denoted by $f(Z)=0$.
\end{definition}
Thus,
\begin{center}
	$f(Z)=0$ if and only if $f\in \langle Z\rangle$,
\end{center}
where the two-sided ideal generated by $x_1-z_1,\dots, x_n-z_n$ is simply denoted by $\langle Z\rangle$, i.e., 
\begin{equation}
\langle Z\rangle:=\langle x_1-z_1, \dots, x_n-z_n\rangle.
\end{equation}	
\begin{definition}\label{definition4.14}
	Let $f\in A$. The \textbf{\textit{vanishing set}} of $f$, also called the \textbf{\textit{set of roots}} of $f$, is denoted by $V(f)$, and defined by  
	\begin{equation}
	V(f):=\{Z\in R^n\mid f(Z)=0\}.
	\end{equation}
	If $S\subseteq A$, then 
	\begin{equation}
	V(S):=\{Z\in R^n\mid f(Z)=0,\ \text{for every $f\in S$}\}.
	\end{equation}
	A subset $X\subseteq R^n$ is \textbf{algebraic} if either $X=R^n$ or there exists $g\neq 0\in A$ such that $X\subseteq V(g)$.
\end{definition} 

\subsection{Algebraic sets and ideals of points}

Some classical properties of affine algebraic sets of commutative algebraic geometry (see \cite{Fulton}, Chapter 1) will be extended in this subsection. 

\begin{theorem}\label{theorem3.4}
	\begin{enumerate}
		\item[\rm (i)]Let $f,g,h\in A$ and $Z:=(z_1,\dots,z_n)\in R^n$. 
		\begin{enumerate}
			\item[\rm (a)]If $f(Z)=0=g(Z)$, then $(f+g)(Z)=0$.
			\item[\rm (b)]$V(f)\subseteq V(gfh)$.
		\end{enumerate}
		\item[\rm (ii)]Let $I:=Ag$ be a left principal ideal of $A$. Then, $V(I)=V(g)$. The same is true for right and two-sided principal ideals of $A$.
		\item[\rm (iii)] 
		\begin{enumerate}
			\item[\rm (a)]$V(0)=R^n$. 
			\item[\rm (b)]$\emptyset$ is algebraic.
			\item[\rm (c)]If $S\subseteq T\subseteq A$, then $V(T)\subseteq V(S)$.
			\item[\rm (d)]If $S\subseteq A$, then $V(S)= V(AS)=V(SA)=V(ASA)$.
			\item[\rm (e)]$V(I)\cup V(J)\subseteq V(I\cap J)$, where $I,J$ are left $($right, two-sided$)$ ideals of $A$.
			\item[\rm (f)]$V(\sum_{k\in \mathcal{K}}I_k)=\bigcap_{k\in \mathcal{K}}V(I_k)$, where $I_k$ is a left left $($right, two-sided$)$ ideal of $A$.
			\item[\rm{(h)}]Let $Z:=(z_1,\dots,z_n)\in R^n$. Then, $\{Z\}\subseteq V(\langle Z\rangle)$. 
		\end{enumerate} 
		\item[\rm (iv)]Let $X\subseteq R^n$. Then,
		\begin{center}
			$I(X):=\{g\in A\mid g(Z)=0, \ \text{for every $Z\in X$}\}$
		\end{center}
		is a two-sided ideal of $A$, called the \textbf{ideal of points} of $X$. Some properties of $I(X)$ are:
		\begin{enumerate}
			\item[\rm (a)]$I(\emptyset)=A$. 
			\item[\rm{(b)}]For $X,Y\subseteq R^n$, $X\subseteq Y\Rightarrow I(Y)\subseteq I(X)$.
			\item[\rm{(c)}]If $I$ is a left $($right, two-sided$)$ ideal of $A$, then $I\subseteq I(V(I))$. 
			\item[\rm{(d)}]$X\subseteq V(I(X))$. 
			\item[\rm{(e)}]If $g\in A$, $V(I(V(g)))=V(g)$. Thus, if $X= V(g)$, then  $V(I(X))=X$.
			\item[\rm{(f)}] $I\left(  V\left(  I\left( X\right)  \right)  \right)  =I\left(
			X\right)$. 
			\item[\rm{(g)}]$I(\bigcup_{k\in \mathcal{K}}X_k)=\bigcap_{k\in \mathcal{K}}I(X_k)$.
			\item[\rm{(h)}]Let $Z:=(z_1,\dots,z_n)\in R^n$. Then, $I(\{Z\})=\langle Z\rangle$. 
		\end{enumerate}
	\end{enumerate}
\end{theorem}
\begin{proof}
	(i) (a) We have $f,g\in \langle Z\rangle$, so $f+g\in \langle Z\rangle$, i.e., $(f+g)(Z)=0$.
	
	(b)	Let $Z\in V(f)$, then $f\in \langle Z\rangle$, then $gfh\in \langle Z\rangle$, i.e., $Z\in V(gfh)$.
	
	(ii) It is clear that $V(I)\subseteq V(g)$. From (i)-(b) we get that $V(g)\subseteq V(I)$.
	
	(iii) (a) It is clear that $V(0)=R^n$. 
	
	(b) Observe that for any $Z:=(z_1,\dots,z_n)\in R^n$, $Z\in V(x_1-z_1+\cdots+x_n-z_n)$. Thus, we have the nonzero polynomial $g:=x_1-z_1+\cdots+x_n-z_n$ and $\emptyset\subseteq V(g)$.
	
	(c) Evident.
	
	(d) Since $S\subseteq AS$, then $V(AS)\subseteq V(S)$; let $Z\in V(S)$ and $g\in AS$, then $g=p_1s_1+\cdots+p_ts_t$, with $p_i\in A$ and $s_i\in S$, $1\leq i\le t$. Since $p_is_i\in \langle Z\rangle$, then $g\in \langle Z\rangle$, so $V(S)\subseteq V(AS)$. Similarly, $V(S)=V(SA)=V(ASA)$.
	
	(e) Since $I\cap J\subseteq I,J$, then $V(I)\cup V(J)\subseteq V(I\cap J)$.
	
	(f) Since $I_k\subseteq \sum_{k\in \mathcal{K}}I_k$ for every $k\in \mathcal{K}$, then $V(\sum_{k\in \mathcal{K}}I_k)\subseteq \bigcap_{k\in \mathcal{K}}V(I_k)$. Let $Z\in \bigcap_{k\in \mathcal{K}}V(I_k)$ and let $g\in \sum_{k\in \mathcal{K}}I_k$, then $g=g_{k_1}+\cdots+g_{k_t}$, with $g_{k_j}\in I_{k_j}$, $1\leq j\leq t$, then from (i)-(a) , $g(Z)=0$, whence $Z\in V(\sum_{k\in \mathcal{K}}I_k)$. Thus, $\bigcap_{k\in \mathcal{K}}V(I_k)\subseteq V(\sum_{k\in \mathcal{K}}I_k)$.
	
	(h) Evident. 	
	
	(iv) (a)-(c) are evident from the definitions.
	
	(d) For $X=\emptyset$ the assertion follows from (a) since $\emptyset\subseteq V(A)$. Let $X\neq \emptyset$. If $Z\in X$, then for every $g\in I(X)$, $g(Z)=0$, and this means that $Z\in V(I(X))$. Therefore, $X\subseteq V(I(X))$. 
	
	(e) From (d), $V(g)\subseteq V(I(V(g)))$. Let $Z\in V(I(V(g)))$, since $g\in I(V(g))$, then $g(Z)=0$, i.e., $Z\in V(g)$. Therefore, $V(I(V(g)))\subseteq V(g)$. 
	
	(f) From (c), $I(X)\subseteq I(V(I(X)))$. From (d), $X\subseteq V(I(X))$, so from (b), $I( V(I(X))\subseteq I(X)$. 
	
	(g) Since $X_k\subseteq \bigcup_{k\in \mathcal{K}}X_k$ for every $k\in \mathcal{K}$, then $I(\bigcup_{k\in \mathcal{K}}X_k)\subseteq I(X_k)$, so $I(\bigcup_{k\in \mathcal{K}}X_k)\subseteq \bigcap_{k\in \mathcal{K}}I(X_k)$. Let $g\in \bigcap_{k\in \mathcal{K}}I(X_k)$ and let $Z\in \bigcup_{k\in \mathcal{K}}X_k$, then there exists $k\in \mathcal{K}$ such that $Z\in X_k$, then $g(Z)=0$, whence $g\in I(\bigcup_{k\in \mathcal{K}}X_k)$. 
	
	(h) Evident.
\end{proof}

Definition \ref{definition4.14} and the previous theorem induces the following consequences.

\begin{corollary}\label{corollary4.5}
	\begin{enumerate}
		\item[\rm (i)]$R^n$ has a \textbf{Zariski topology} where the closed sets are the algebraic sets. 
		\item[\rm (ii)]If $X\subseteq R^n$ is finite, then $X$ is algebraic, and hence, closed.
	\end{enumerate}
\end{corollary}
\begin{proof}
	(i) By Definition \ref{definition4.14}, $R^n$ is algebraic. From Theorem \ref{theorem3.4} we know that $X=\emptyset $ is algebraic. Moreover, let $X\subseteq V(f)$ and $Y\subseteq V(g)$ be algebraic, with $0\neq f\in A$ and $0\neq g\in A$, then, since $A$ is a left noetherian domain, $A$ is a left Ore domain, i.e, $Af\cap Ag\neq 0$, whence 
	\begin{center} 
		$X\cup Y\subseteq V(f)\cup V(g)\subseteq V(l)$,
	\end{center}
	where $0\neq l\in Af\cap Ag$. Finally, let $\{X_k\}_{k\in \mathcal{K}}$ be a family of algebraic sets of $R^n$, then for every $k\in \mathcal{K}$ there exists $0\neq g_k\in A$ such that $X_k\subseteq V(g_k)$, hence
	\[
	\bigcap_{k\in \mathcal{K}} X_k\subseteq \bigcap_{k\in \mathcal{K}}V(g_k)=\bigcap_{k\in \mathcal{K}}V(Ag_k)=V(\sum_{k\in \mathcal{K}}Ag_k)\subseteq V(Ag_k)=V(g_k),\ \text{for any $k$}.
	\]
	
	(ii) We know that $X=\emptyset $ is algebraic. Let $\emptyset \neq X:=\{Z_1,\dots,Z_r\}$, then $I(X)\neq 0$. In fact, let
	\begin{center} 
		$f_i:=(x_1-z_{i1})+\cdots+(x_n-z_{in})$, with $Z_i:=(z_{i1},\dots,z_{in})$, $1\leq i\leq r$.
	\end{center}
	Let $0\neq l\in Af_1\cap \cdots \cap Af_r$. Observe that $l\in I(X)$: Indeed, for every $i$, $l=p_if_i$, for some $p_i\in A$, so $l(Z_i)=p_if_i(Z_i)=0$. This shows that $I(X)\neq 0$. 	
	Thus, $X\subseteq V(I(X))\subseteq V(l)$ is algebraic.
	
\end{proof}

\begin{definition}
	Let $f\in A-R$. $V(f)$ is called the \textbf{skew hypersurface} defined by $f$. In particular,
	\begin{enumerate}
		\item[\rm (i)]$V(f)$ is a \textbf{skew plane curve} if $n=2$. 
		\item[\rm (ii)]$V(f)$ is a \textbf{skew hyperplane} if $\deg(f)=1$, i.e., $f=r_0+r_1x_1+\cdots +r_nx_n$, with $r_i\in A$, $0\leq i\leq n$. When $n=2$, $V(f)$
		is a \textbf{skew line}.
	\end{enumerate}
\end{definition}

\begin{corollary}\label{corollary4.20}
	Let $I$ be a left ideal of $A$. Then,
	\begin{center}
		$V(I)=V(f_1)\cap \cdots \cap V(f_r)$, where $I=Af_1+\cdots+Af_r$. 
	\end{center}
	Thus, if $f_i\in A-R$, for $1\leq i\leq r$, then $V(I)$ is a finite intersection of skew hypersurfaces.
\end{corollary}
\begin{proof}
	For $I=0$, $r=1$ and $f_1=0$. Let $I\neq 0$, since $A$ is left noetherian, there exist $f_1,\dots, f_r\in A$ such that $I=Af_1+\cdots+Af_r$. From Theorem \ref{theorem3.4},
	\begin{center}
		$V(I)=V(Af_1+\cdots+Af_r)=V(Af_1)\cap \cdots \cap V(Af_r)=V(f_1)\cap \cdots \cap V(f_r)$.
	\end{center}    
\end{proof}

\begin{remark}\label{remark4.21}
	(i) There exist skew $PBW$ extensions such that $V(A)\neq \emptyset$. In fact, let $A:=\sigma(\mathbb{Q})\langle x,y,z\rangle$ defined by
	\begin{center}
		$yx=xy-1$, $zx=xz$, $zy=yz$. 
	\end{center}
	Consider the left ideal $I:=A(x-1)+Ay+Az$ and observe that
	\begin{center}
		$1=-y(x-1)+(x-1)y+0z=-y(x-1)+(x-1)(y-0)+0(z-0)$,
	\end{center}
	i.e., $I=A$ and $(1,0,0)\in V(1)=V(A)$.
	
	(ii) According to Corollary \ref{corollary4.5}, if $R$ is finite, then $R^n$ is algebraic, and hence we do not need to assume this condition on $R^n$ in Definition \ref{definition4.14}. But if $R$ is infinite, we can not assert that there is $0\neq g\in A$ such that $R^n\subseteq V(g)$. Consider for example that $A:=\mathbb{F}[x_1,\dots,x_n]$ is the commutative ring of polynomials with coefficients in an infinite field $\mathbb{F}$, then $I(\mathbb{F}^n)=0$ (see \cite{Fulton}, Chapter 1) and contrary assume that there exists $0\neq g\in A$ such that $\mathbb{F}^n\subseteq V(g)$, hence $I(V(g))\subseteq I(\mathbb{F}^n)=0$, but $g\in I(V(g))$, a contradiction.
\end{remark}

We conclude this subsection with a corollary that complements the part (i) of Remark \ref{remark4.21} for the particular case of classical polynomials of one single variable, but with ring of coefficients a little bit more general. For this, we give first a proposition that describes the normal elements of quasi-commutative skew $PBW$ extensions. 
Recall that $f\in A$ is \textbf{\textit{normal}} if $Af=fA$. 

\begin{proposition}\label{proposition4.22}
	Assume that $A$ is quasi-commutative.
	\begin{enumerate}
		\item[\rm (i)]Let $f=cx^{\alpha}h\in A$, where $c\in R^*$ is normal in R, $x^{\alpha}\in {\rm Mon}(A)$ and $h\in Z(A)$. Then, $f$ is normal.
		\item[\rm (ii)]Let $f=c_1x^{\alpha_1}+\cdots
		+c_tx^{\alpha_t}\in A$, with $c_i\in R-\{0\}$, $1\leq i\leq t$, and $x^{\alpha_1}\succ \cdots \succ
		x^{\alpha_t}$. If $f$ is normal, then $c_i$ is normal in $R$, for every $1\leq i\leq t$.
	\end{enumerate}   
\end{proposition}
\begin{proof}
	$\succeq$ is the ${\rm deglex}$ order on ${\rm Mon}(A)$, but any other monomial order could be used (for other monomial orders see \cite{Lezama-sigmaPBW}, Chapter 13).
	
	(i) Since the product of normal elements is normal and clearly  $h$ is normal, then we have to show only that $c$ and $x^{\alpha}$ are normal elements of $A$. Let $a:=a_1x^{\beta_1}+\cdots
	+a_sx^{\beta_s}\in A$,  with $a_j\in R-\{0\}$, $1\leq j\leq s$, and $x^{\beta_1}\succ \cdots \succ
	x^{\beta_s}$. We have $ac=(a_1x^{\beta_1}+\cdots
	+a_sx^{\beta_s})c=a_1x^{\beta_1}c+\cdots
	+a_sx^{\beta_s}c=a_1\sigma^{\beta_1}(c)x^{\beta_1}+\cdots+a_s\sigma^{\beta_s}(c)x^{\beta_s}=a_1cc^{-1}\sigma^{\beta_1}(c)x^{\beta_1}+\cdots +a_scc^{-1}\sigma^{\beta_s}(c)x^{\beta_s}=c(a_1c^{-1}\sigma^{\beta_1}(c)x^{\beta_1}+\cdots +a_sc^{-1}\sigma^{\beta_s}(c)x^{\beta_s})$, thus $Ac\subseteq cA$. Since $A$ is bijective, then we can prove similarly that $cA\subseteq Ac$. Now, $ax^{\alpha}=(a_1x^{\beta_1}+\cdots
	+a_sx^{\beta_s})x^{\alpha}=a_1x^{\beta_1}x^{\alpha}+\cdots+a_sx^{\beta_s}x^{\alpha}=a_1c_1'x^{\alpha}x^{\beta_1}+\cdots+a_sc_s'x^{\alpha}x^{\beta_s}=x^{\alpha}\sigma^{-\alpha}
	(a_1c_1')x^{\beta_1}+\cdots+x^{\alpha}\sigma^{-\alpha}(a_sc_s')x^{\beta_s}$, for some $c_j'\in R^*$, $1\leq j\leq s$, thus $Ax^{\alpha}\subseteq x^{\alpha}A$. In a similar way we can prove that $x^{\alpha}A\subseteq Ax^{\alpha}$.
	
	(ii) Let $r\in R-\{0\}$, then $rf\in Af=fA$, so $rf=fg$, for some $g\in A$. Since $A$ is a domain, $\deg(rf)=\deg(f)=\deg(fg)=\deg(f)+\deg(g)$, so $g\in R-\{0\}$, but as $A$ is quasi-commutative, for every $1\leq i\leq t$, $rc_i=c_i\sigma^{\alpha_i}(g)$. This means that $Rc_i\subseteq c_iR$. Considering now $fr\in fA=Af$ we get that $fr=hf$, with $h\in R-\{0\}$, so for every $i$, $c_i\sigma^{\alpha_i}(r)=hc_i$, but since $A$ is bijective, $\sigma^{\alpha_i}(R)=R$, and hence $c_iR\subseteq Rc_i$. Thus, $c_iR= Rc_i$, i.e., $c_i$ is normal in $R$.  
\end{proof} 

\begin{corollary}\label{proposition4.23}
	Let $S$ be a left noetherian domain and $B:=S[x]$ be the habitual ring of polynomials.
	\begin{enumerate}
		\item[\rm (i)]Let $f\in B$, with ${\rm lc}(f)\in S^*$. $f$ is a normal polynomial if and only if $f=cx^th$, where $c\in S^*$ is normal, $t\geq 0$ and $h\in Z(B)$.
		\item[\rm (ii)]Let $\mathbb{F}$ be an algebraically closed field and assume that $S$ is an $\mathbb{F}$-algebra with trivial center.
		Let $I=Bf_1+\cdots+Bf_r$ be a left ideal of $B$, where $0\neq f_i$ is normal and ${\rm lc}(f_i)\in S^*$, for every $1\leq i\leq r$. If $V(I)=\emptyset$, then $I=B$.  
	\end{enumerate}
\end{corollary}
\begin{proof}
	Notice first that $B$ is a quasi-commutative bijective skew $PBW$ extension of $S$, so we can use all the previous results.
	
	(i) $\Rightarrow )$ Let $f:=f_0+f_1x\cdots +f_nx^n\neq 0$, with $f_n\neq 0$. We can assume that $f$ is monic. In fact, $f=f_n(f_n^{-1}f_0+f_n^{-1}f_1x+\cdots+x^n)=f_nf'$, with $f':=f_n^{-1}f_0+f_n^{-1}f_1x+\cdots+x^n$. We have $Bf=fB$, but from Proposition \ref{proposition4.22}, $f_n$ is a normal element of $S$, so  $Bf_n=f_nB$, and hence $Bf_nf'=f_nf'B=f_nBf'$, but since $B$ is a domain, then $f'B=Bf'$. Therefore, if we prove the claimed for monic polynomials, then $f'=c'x^th$, with $c'\in S^*$ normal, $t\geq 0$ and $h\in Z(B)$, so $f=cx^th$, with $c:=f_nc'\in S^*$ normal.  
	
	We will prove the claimed by induction on $n$. If $n=0$, then $c=f=1\in S^*$ normal, $t=0$ and $h=1$. If $n=1$, then $f=f_0+x$; if $f_0=0$, then we get the claimed with $c=1$, $t=1$ and $h=1$. Assume that $f_0\neq 0$; let $s\in S$, then $fs=bf$, with $b\in B$, this implies that $b:=b_0\in S$, and hence $f_0s+sx=b_0f_0+b_0x$, whence $s=b_0$ and $f_0s=sf_0$, i.e., $f_0\in Z(S)$. Therefore, $f\in Z(B)$ and we get the claimed with $c=1$, $t=0$ and $h=f$. This completes the proof of case $m=1$.
	
	Assume the claimed proved for non zero monic normal polynomials of degree $\leq n-1$. We will consider two possible cases. 
	
	\textit{Case 1}. $f_0\neq 0$. As before, let $s\in S$, then $fs=bf$, with $b\in B$, but this implies that $b:=b_0\in S$ and $f_is=sf_i$ for every $0\leq i\leq n-1$. Thus, $f\in Z(B)$ and we get the claimed with $c=1$, $t=0$ and $h=f$.
	
	\textit{Case 2}. $f_0=0$. Let $l$ be minimum such that $f_l\neq 0$. Then $f=x^{l}f'$, with $l\geq 1$ and $0\neq f'\in A$ monic. Observe that $x^l$ and $f'$ are normal: In fact, $x^lB=Bx^l$ and $Bf=Bx^lf'=x^lBf'=x^lf'B=fB$, but since $B$ is a domain, then $Bf'=f'B$. By induction, there exist $c'\in S^*$ normal, $l'\geq 0$ and $h\in Z(B)$ such that $f'=c'x^{l'}h$, hence $f=x^{l}c'x^{l'}h=c'x^{l+l'}h$, so we obtain the claimed with $c:=c'$ and $t:=l+l'$.
	
	$\Leftarrow )$ This follows from the previous proposition. 
	
	(ii) From Corollary \ref{corollary4.20} we have that $V(I)=V(f_1)\cap \cdots \cap V(f_r)$, but by (i), for every $1\leq i\leq r$, $f_i=c_ix^{t_i}h_i$, where $c_i\in S^*$ is normal, $t_i\geq 0$ and $h_i\in Z(B)=Z(S)[x]=\mathbb{F}[x]$. Since $c_i\in S^*$, then $V(f_i)=V(x^{t_i}h_i)$, hence
	\begin{center} 
		$V(I)\supseteq V(x^{t_1}h_1)\cap \cdots \cap V(x^{t_r}h_r)\supseteq V(d)$, where $d:={\rm gcrd}(x^{t_1}h_1,\dots,x^{t_r}h_r)$ in $\mathbb{F}[x]$. 
	\end{center}
	Since $V(I)=\emptyset$, then $V(d)=\emptyset$ with respect to $R^n$, whence, $V(d)=\emptyset$ with respect to $\mathbb{F}^n$, but $\mathbb{F}$ is algebraically closed, then $d\in \mathbb{F}^*$. We have $d=g_1x^{t_1}h_1+\cdots+g_rx^{t_r}h_r$, for some $g_1,\dots,g_r\in \mathbb{F}[x]$. Let $c:=c_1\cdots c_r\in S^*$, since $c_i$ is normal and $g_i\in Z(B)$, for every $1\leq i\leq r$, we get $cd=g_1'c_1x^{t_1}h_1+\cdots+g_r'c_rx^{t_r}h_r$, where every $g_i'\in B$. Thus, $cd=g_1'f_1+\cdots+g_r'f_r\in I\cap S^*$, i.e., $I=B$.  
\end{proof}

\subsection{Hilbert's Nullstellensatz theorem for skew $PBW$ extensions}

In this subsection we prove the main result of the paper, namely, to give an adaptation of the Hilbert's Nullstellensatz theorem for skew $PBW$ extensions such that it covers the classical theorem of commutative algebraic geometry. For this purpose we need the following preliminary lemma supported in the Gröbner theory of skew $PBW$ extensions.

\begin{lemma}\label{lemma4.15}	Let $A:=\sigma(\mathbb{F})\langle x_1,\dots,x_n\rangle$ be a quasi-commutative bijective skew $PBW$ extension of $\mathbb{F}$, where $\mathbb{F}$ is a field. Then, for every $Z:=(z_1,\dots,z_n)\in \mathbb{F}^n$, $\langle Z\rangle$ is completely semiprime. 
\end{lemma}
\begin{proof}
	We have to show first that $\langle Z\rangle\neq A$: Contrary, assume that $1\in A$, then 
	\begin{center}
		$1=p_1(x_{i_1}-z_{i_1})q_1+\cdots+p_m(x_{i_m}-z_{i_m})q_m$, with $p_j,q_j\in A$ and $i_j\in \{1,\dots,n\}$, $1\leq j\leq m$.
	\end{center}
	By Proposition \ref{algdivforPBW},
	\begin{center} 
		$q_j=q_j'(x_{i_j}-z_{i_j})+h_j$, with $q_j',h\in A$, $h_j$ reduced w.r.t. $x_{i_j}-z_{i_j}$, $1\leq j\leq m$.
	\end{center}
	So, 
	\begin{center}
		$1=p_1(x_{i_1}-z_{i_1})(q_1'(x_{i_1}-z_{i_1})+h_1)+\cdots+p_m(x_{i_m}-z_{i_m})(q_m'(x_{i_m}-z_{i_m})+h_m)=p_1(x_{i_1}-z_{i_1})q_1'(x_{i_1}-z_{i_1})+p_1(x_{i_1}-z_{i_1})h_1+
		\cdots+p_m(x_{i_m}-z_{i_m})q_m'(x_{i_m}-z_{i_m})+p_m(x_{i_m}-z_{i_m})h_m$.
	\end{center}
	
	Since every $h_j$ is reduced, then $h_j\in \mathbb{F}$, but $A$ is quasi-commutative, then $1\in I$, where $I$ is a left ideal of $A$ generated by elements of the form  $c_{i_j}x_{i_j}-z_{i_j}'$, where $c_{i_j},z_{i_j}'\in \mathbb{F}$ with $c_{i_j}\neq 0$, $1\leq j\leq m$ (actually, in some cases $c_{i_j}=1$ and $z_{i_j}'=z_{i_j}$, and in other cases, when $h_j\neq 0$, then $c_{i_j}=\sigma_{i_j}(h_j)$ and $z_{i_j}'=z_{i_j}h_j$, where $\sigma_{i_j}$ is as in Proposition \ref{sigmadefinition}). It is clear from Definition \ref{definition3.10} that the generators $c_{i_j}x_{i_j}-z_{i_j}'$ of $I$ conform a Gröbner basis of $I$. Since $1\in I$, from the part (iii) of Proposition \ref{153} we get that $x_{i_j}$ divides $1$ for some $j$, a contradiction. Hence, $\langle Z\rangle\neq A$.     
	
	Now, let $f\in A$ such that $f^2\in \langle Z\rangle$. Applying again Proposition \ref{algdivforPBW}, there exist
	polynomials $p_1,\dots ,p_t,h\in A$, with remainder $h$ reduced w.r.t.\ $F:=\{x_1-z_1,\dots,x_n-z_n\}$, such that
	\begin{center}
		$f=p_1(x_1-z_1)+\cdots +p_n(x_n-z_n)+h$.
	\end{center}
	Since $h$ is reduced, then $h\in \mathbb{F}$. If $h=0$, then $f\in \langle Z\rangle$ and the proof is over. Assume that $h\neq 0$, then
	\begin{center} 
		$h^2=(f-[p_1(x_1-z_1)+\cdots +p_n(x_n-z_n)])^2\in \langle Z\rangle$,
	\end{center} 
	hence $\langle Z\rangle=A$, a contradiction.
\end{proof} 

\begin{theorem}[\textbf{Hilbert's Nullstellensatz}]\label{Nullstellensatz}
	Let $A:=\sigma(\mathbb{F})\langle x_1,\dots,x_n\rangle$ be a quasi-commutative bijective skew $PBW$ extension of $\mathbb{F}$, where $\mathbb{F}$ is an algebraically closed field. Assume that $Z(A)$ is a polynomial ring in $n$ variables with coefficients in $\mathbb{F}$.  
	Let $I$ be a two-sided ideal of $A$. Then,  
	\begin{center}
		$\langle I_{Z(A)}(V_{Z(A)}(J))\rangle\subseteq \sqrt{I}\subseteq I(V(I))$,
	\end{center}
	where $J:=l^{-1}(I)$, $l:Z(A)\to A$ is the inclusion of the center of $A$ in $A$, $V_{Z(A)}(J)$ is the vanishing set of $J$ with respect to $Z(A)$ and $I_{Z(A)}(V_{Z(A)}(J))$ is the ideal of points of $V_{Z(A)}(J)$ with respect to $Z(A)$. 
\end{theorem}
\begin{proof}
	
	$\sqrt{I}\subseteq I(V(I))$: If $V(I)=\emptyset$, then $I(V(I))=A$ and hence there is nothing to prove. Assume that $V(I)\neq \emptyset$ and let $f\in \sqrt{I}$, then $f$ is $I$-strongly nilpotent, and hence, $I$-nilpotent, so there exists $m\geq 1$ such that $f^m\in I$. Let $Z:=(z_1,\dots,z_n)\in V(I)$, then $f^m\in \langle Z\rangle$. From Lemma \ref{lemma4.15}, $f\in \langle Z\rangle$, i.e., $f\in I(V(I))$. 
	
	$\langle I_{Z(A)}(V_{Z(A)}(J))\rangle\subseteq \sqrt{I}$: Consider the inclusion $Z(A)\xrightarrow{l} A$ and let $J:=l^{-1}(I)$. Then, $J=l(J)=l(l^{-1}(I))\subseteq I$, and let $\langle J\rangle:=AJA$ be the two-sided ideal of $A$ generated by $J$. We have $\langle J\rangle\subseteq I$, so $\sqrt{\langle J\rangle}\subseteq \sqrt{I}$, but $\langle \sqrt{J}\rangle\subseteq \sqrt{\langle J\rangle}$, where $\sqrt{J}$ is the radical of $J$ in the ring $Z(A)$. In fact, let $w\in \sqrt{J}$, then there exists $m\geq 1$ such that $w^m\in J\subseteq \langle J\rangle$, but $w\in Z(A)$, then $w$ is $\langle J\rangle$-strongly nilpotent, i.e., $w\in \sqrt{\langle J\rangle}$. Thus, $\langle \sqrt{J}\rangle\subseteq \sqrt{I}$. Applying the classical Hilbert's Nullstellensatz for $Z(A)$ (here we use that $\mathbb{F}$ is algeraically closed) we have $\sqrt{J}=I_{Z(A)}(V_{Z(A)}(J))$, so we get that 
	$\langle I_{Z(A)}(V_{Z(A)}(J))\rangle\subseteq \sqrt{I}$. 
\end{proof}

\begin{example}\label{example3.13}
	Next we present some concrete examples of skew $PBW$ extensions that satisfy the hypothesis of Theorem \ref{Nullstellensatz}. $\mathbb{F}$ denotes an algebraically closed field.
	
	(i) It is clear that if $A=\mathbb{F}[x_1,\dots,x_n]$ and $I$ is an ideal of $A$, then in Theorem \ref{Nullstellensatz} we have $\langle I_{Z(A)}(V_{Z(A)}(J))\rangle=I(V(I))$, and hence, $I(V(I))=\sqrt{I}$. 
	
	(ii) If $q\neq 1$ is an arbitrary root of unity of degree $m\geq 2$, then the center of the quantum
	plane $A:=\mathbb{F}_{q}[x,y]$ is the subalgebra generated by $x^m$ and $y^m$, i.e.,
	$Z(\mathbb{F}_{q}[x,y])=\mathbb{F}[x^m,y^m]$ (see \cite{Shirikov} or also \cite{Lezama-sigmaPBW}, Proposition 3.3.14). Recall that the rule of multiplication in $A$ is given by $yx=qxy$.
	
	(iii) The previous example can be generalized in the following way (see \cite{Zhangetal}, Lemma 4.1, or also \cite{Lezama-sigmaPBW}, Proposition 3.3.15): Let $q\in \mathbb{F}-\{0\}$ and $A:=\mathbb{F}_q[x_1,\dots,x_n]$ be the skew $PBW$ extension
	defined by $x_jx_i=qx_ix_j$ for all $1\leq i<j\leq n$. If $n\geq 2$ and $q\neq 1$ is a root of
	unity of degree $m\geq 2$, then
	\begin{enumerate}
		\item[\rm (a)]If $q=-1$, then
		\begin{center}
			$Z(A)=\mathbb{F}[x_1^2,\dots,x_n^2]$ when $n$ is even.
			
		\end{center}
		\item[\rm (b)]If $q\neq -1$, then
		\begin{center}
			$Z(A)=\mathbb{F}[x_1^m,\dots,x_n^m]$ when $n$ is even.
			
		\end{center}
	\end{enumerate}
	
	(iv) Consider that for every $1\leq i,j\leq n$, $q_{ij}\in \mathbb{F}-\{0\}$ is a nontrivial
	root of unity of degree $d_{ij}<\infty$ and let $A:=\mathbb{F}_{\boldsymbol{\rm q}}[x_1,\dots,x_n]$ be the skew $PBW$ extension
	defined by $x_jx_i=q_{ij}x_ix_j$ for all $1\leq i<j\leq n$. Let $k_{ij}\in \mathbb{Z}$ such that $|k_{ij}|<d_{ij}$,
	${\rm lcd}(k_{ij}, d_{ij})=1$ and $q_{ij}=\exp(2\pi\sqrt{-1}\frac{k_{ij}}{d_{ij}})$ $($choosing
	$k_{ji}:=-k_{ij}$$)$. Let $L_i:={\rm lcm}\{d_{ij}|j=1,\dots,n\}$. Then 
	$Z(\mathbb{F}_{\boldsymbol{\rm q}}[x_1,\dots,x_n])$ is a polynomial ring if and only if it is of the form
	$\mathbb{F}[x_1^{L_1},\dots,x_n^{L_n}]$ (see \cite{Zhangetal3}, Theorem 0.3 or also \cite{Lezama-sigmaPBW}, Proposition 3.3.17).
\end{example}

We conclude the paper with the following problem induced by Proposition \ref{proposition4.22}.

\begin{problem}
Describe all normal elements of an arbitrary skew $PBW$ extension.
\end{problem}

%\begin{center}
%\textbf{Acknowledgements}
%\end{center}

%%%%%%%%%%%%%%%%%%%%%%%%%%%%%%%%%%%%%%%%%%%%%%%%%
%%%%%%%%%%%%%%%%%%%%%%%%%%%%%%%%%%%%%%%%%%%%%%%%%

\end{document}